\documentclass[12pt,reqno]{amsart}
\usepackage[left=3cm, right=3cm, top=2.8cm, bottom=2.8cm]{geometry}
\usepackage{amsmath,amssymb,amsfonts,algorithm2e}
\usepackage[utf8]{inputenc}
\usepackage[T1]{fontenc}
\usepackage[english]{babel}
\usepackage{graphicx}

\newtheorem*{remark*}{Remark}
\newtheorem{theorem}{Theorem}
\newtheorem{lemma}[theorem]{Lemma}

\newtheorem{corollary}[theorem]{Corollary}
\newtheorem{definition}[theorem]{Definition}

\newtheorem{example}[theorem]{Example}

\begin{document}

\title[A Polylogarithm Solution to the Epsilon--Delta Problem]{A Polylogarithm Solution to the Epsilon--Delta Problem}

\author[P. M. Carvalho-Neto]{Paulo M. de Carvalho-Neto}
\address[Paulo M. de Carvalho-Neto]{Departamento de Matem\'atica, Centro de Ciências Físicas e Matemáticas, Universidade Federal de Santa Catarina, Florian\'{o}polis - SC, Brazil}
\email[P. M. Carvalho-Neto]{paulo.carvalho@ufsc.br}
\author[P. A. Liboni Filho]{Paulo A. Liboni Filho}
\address[Paulo A. Liboni Filho]{Departamento de Matem\'{a}tica, Centro de Ci\^{e}ncias Exatas, Universidade Estadual de Londrina, 86057-970, Londrina, PR, Brazil}
\email[P. A. Liboni Filho]{liboni@uel.br }


\begin{abstract}  Let $f$ be a continuous real function defined in a subset of the real line. The standard definition of continuity at a point $x$ allow us to correlate any given epsilon with a (possibly depending of $x$) delta value. This pairing is known as the epsilon--delta relation of $f$. In this work, we demonstrate the existence of a privileged choice of delta in the sense that it is continuous, invertible, maximal and it is the solution of a simple functional equation. We also introduce an algorithm that can be used to numerically calculate this map in polylogarithm time, proving the computability of the epsilon--delta relation. Finally, some examples are analyzed in order to showcase the accuracy and effectiveness of these methods, even when the explicit formula for the aforementioned privileged function is unknown due to the lack of analytical tools for solving the functional equation.
\end{abstract}

\maketitle

\section{Introduction}

It was L. Kronecker who first coined the expression ``Arithmetization of Analysis'', which eventually became the standard name to designate a group of important research activities carried out during the second half of the 19th century. The program, which is commonly considered completed by 1872, lead to core results in the foundations of mathematics, such as the construction of the real numbers and the definition of limit (cf. \cite{BoMe1,Gr1,Le1}).

The arithmetization marks a paradigm shift in mathematical proofs, with the abandonment of geometric intuition as criteria of truth in favor of a more logical and theoretical reasoning. Eventually, this new framework made all the modern definitions and theorems possible, giving birth to a new aeon in analysis and mathematics.

In the preceding context, B. Bolzano and A. L. Cauchy are considered to be the first to formally discuss the abstract concept of continuity using the $\varepsilon$-$\delta$ definition between the years of 1817 and 1823 (cf. \cite{BoKa1,Fe1}). This formulation allowed the community to address the continuity conceptualization in more abstract spaces, like metric ones (cf. \cite{Di1}). For the sake of clarity, permit us to briefly recall it. Let $(M_1,d_{M_1})$ and $(M_2,d_{M_2})$ be metric spaces. Given a function $f:M_1\rightarrow M_2$ and a point $x\in M_1$, we say that $f$ is \emph{continuous} at $x$ if for any $\varepsilon>0$ there exists $\delta>0$ such that \[ y\in M_1\textrm{ and }d_{M_1}(x,y)<\delta\Longrightarrow d_{M_2}(f(x),f(y))<\varepsilon. \]

Conceptually, the previous definition is prescribed by an implication. Such formalization, even been well known by the entire community, settles an obstacle to directly verifying if a given function is continuous at a fixed point $x\in M_1$. Therefore, instead of using the definition itself, it is usual to apply theorems about continuity to address this matter --- like those who ensures that this property is preserved by products, compositions and linear operations.

The main drawback of using such results is not knowing at least one of the possible $\delta=\delta(x,\varepsilon)$ for the continuous function in question. As pointed in the literature (cf. \cite{CaLi1}), explicitly knowing it can be useful specially when the function is not differentiable and you want to use an inequality like the one provided by the Mean Value Theorem. It is noteworthy that presenting such inequalities for a variety of spaces and functions is an endless endeavor in analysis; see for instance \cite{An1,ClLe1,ClLe2,FeWaHaXiTu1,So2,So1,Tu1} and references therein.

This outcome marks the starting point of this work. More specifically, we focus in the discussion of the $\delta=\delta(x,\varepsilon)$ relation of a continuous function by presenting some new results that implies the existence of a continuous choice of $\delta(x,\varepsilon)$ which is invertible, maximal and can be evaluated by a simple and computable functional equation.

These results are then used to answer the open question about the possibility to numerically determinate the $\varepsilon$--$\delta$ relation of a continuous function within a prescribed precision. We also present and discuss a sample algorithm that uses our computable equation to solve this issue in polylogarithm time.

Bearing last observations in mind, we now present the structure of this paper. Section \ref{sec:theory} introduces the formalism to proof the existence of such privileged choice of $\delta$. It also recalls the conditions on $f$ that allow us to guarantee the existence of a maximum $\delta$, indicated by $\Pi_x^f(\varepsilon)$, that suits the continuity definition of $f$.

By letting $(x, \varepsilon)$ vary in a suitable $\Omega \subset M_1 \times \mathbb{R}^+$, we prove that the relation $(x, \varepsilon) \mapsto \Pi_x^f(\varepsilon)$ defines a continuous, invertible and maximal operator, which also satisfies a computable equation. This gives rise to a well behaved map called the \emph{continuity function} for $f$. Moreover, its graph is a manifold that is called the \emph{$\varepsilon$--$\delta$ manifold} for $f$.

Our second goal, which is discussed in Section \ref{sec:algorithm}, is to use the theorems we derived in the previous section to address the computability of the continuity function and provide polylogarithm pseudocode for a sample implementation. Besides that, we also address hypotheses for numerical computation within a prescribed precision.

Finally, Section \ref{sec:conclusions} presents the continuity function and the $\varepsilon$--$\delta$ manifold for three mappings of the following classes: exponential, rational and affine functions --- all of these found both analytically and numerically, in the fashion we established before. After that, we present an example where the continuity function is not explicitly known.

\section{Theoretical Foundation}
\label{sec:theory}

This section is devoted to discuss the definitions, notations and results that are used throughout this work. Therewith, assume that $(M_1, d_{M_1})$ and $(M_2, d_{M_2})$ denote metric spaces. We also convene that the open ball in $M_i$, with center $x\in M_i$ and radius $r>0$, is denoted by  $B_{M_i}(x,r)$. Besides that, for any subsets $A\subset M_1$ and $B\subset M_2$, it is assumed that the notation $\mathcal{F}(A, B)$ refers to the collection of all functions $f:A\rightarrow B$. If there is no risk of confusion, we simply write $\mathcal{F}$ to denote such entity.

Let us begin this section briefly recalling some preliminary tools and ideas that were addressed and proved in \cite{CaLi1}.

\begin{definition}
An element $(f, x, \varepsilon)\in\mathcal{F}(M_1, M_2) \times M_1 \times \left( \mathbb{R}^+\setminus\{0\}\right)$ is called a \emph{triplet} associated to $M_1$ and $M_2$. When there is no risk of confusion, we simply say that $(f, x, \varepsilon)$ is a \emph{triplet}.
\end{definition}

\begin{definition}
A positive real number $\delta$ is said to be \emph{suitable} for a given triplet $(f,x,\varepsilon)$ if
\[
    f\left(B_{M_1}(x,\delta)\right) \subset B_{M_2}(f(x), \varepsilon).
\] Furthermore, the set of all suitable positive real numbers for that triplet is denoted by $\Delta_{f,x}(\varepsilon)$.
\end{definition}

Note that we are not making any assumptions on $f$ at this point. Nevertheless, it is important to note that the previous definition renders the idea that a certain number $\delta$ fits the continuity definition for a function $f$ at a fixed point $x$, for a particular choice of $\varepsilon$.

\begin{definition}
Given $x\in M_1$ and a function $f:M_1\rightarrow M_2$, define
\[
    E_f(x) =\{\varepsilon>0:\Delta_{f,x}(\varepsilon)\textrm{ is a non empty, bounded set}\}.
\]
\end{definition}

It is not difficult to see that $\Delta_{f,x}(\varepsilon)$ and $E_f(x)$ are intimately connected to the continuity properties of $f$. For instance, $f$ is continuous at $x$ if and only if $\Delta_{f,x}(\varepsilon) \not = \emptyset$ for any positive value $\varepsilon$. Moreover, a necessary and sufficient condition for $f$ to be uniformly continuous is that $\cap_{x \in M_1} \Delta_{f,x}(\varepsilon)\neq\emptyset$ for any positive value $\varepsilon$. The following example gives one possible scenario for the aforesaid sets.

\begin{example}
\label{ex:log}
Consider $M_1=\mathbb{R}^+$, $M_2=\mathbb{R}$ and $\{d_i\}_{i=1,2}$ the real Euclidean metrics. If $f:M_1\rightarrow M_2$ is the natural logarithm function, then for any $\varepsilon>0$ and $x\in M_1$ we obtain that $\Delta_{f,x}(\varepsilon)=(0, x-xe^{-\varepsilon}]$. This allow us to conclude that $E_{f}(x)=(0,\infty)$ for any $x\in M_1$. On the other hand, since $\cap_{x \in M_1} \Delta_{f,x}(\varepsilon)=\emptyset$ we conclude that $f$ is not uniformly continuous.
\end{example}

From the above considerations, we point a couple of interesting results that better describe all the possible topological configurations for the sets discussed above. The proofs of these theorems can be found in \cite{CaLi1}.

\begin{theorem}
\label{thm:topology}
If $(f,x,\varepsilon)$ is a triplet, then one, and only one, of the following alternatives occurs.
\begin{itemize}
    \item[(i)] $\Delta_{f,x}(\varepsilon)=\emptyset$;
    \item[(ii)] $\Delta_{f,x}(\varepsilon)=(0,\infty)$;
    \item[(iii)] There is a certain $\delta>0$ such that $\Delta_{f,x}(\varepsilon)=(0,\delta]$.
\end{itemize}
\end{theorem}

\begin{theorem} Assume that $(f,x,\varepsilon)$ is a triplet such that $f:M_1\rightarrow M_2$ is continuous at $x \in M_1$. Then one, and only one, of the following alternatives occurs.
\begin{itemize}
    \item[(i)] If $f$ is an unbounded function, then $E_{f}(x)=(0,\infty)$;
    \item[(ii)] If $f$ is a constant function, then $E_{f}(x)=\emptyset$;
    \item[(iii)] There is a real number $\varepsilon>0$ such that $E_f(x)=(0,\varepsilon]$ or $E_f(x)=(0,\varepsilon)$.
\end{itemize}
\end{theorem}

We emphasize that if $f:M_1\rightarrow M_2$ is a continuous and non-constant function, then $E_f(x)$ is a non-empty set for all $x$. Also remember that if $\varepsilon \in E_f(x)$, then $\Delta_{f,x}(\varepsilon)$ is a non-empty and bounded set. Therefore, Theorem \ref{thm:topology} implies that $\Delta_{f,x}(\varepsilon)$ has a maximum value. This fact is essential in the next discussion.

Now that all the main formal requirements are already presented, let us address the conceptualization of the continuity function associated to $f$.

\begin{definition} Given a function $f:M_1\rightarrow M_2$ and a point $x\in M_1$ such that $E_{f}(x) \not = \emptyset$, define the \emph{continuity function} $\Pi^f_x:  E_{f}(x) \rightarrow  (0,\infty)$ by
\[
    \Pi^f_x(\varepsilon)= \max \Delta_{f,x}(\varepsilon).
\]
\end{definition}

For an example of continuity function, recall Example \ref{ex:log}. In this case, observe that $\Pi_{x}^f(\varepsilon)=x(1-e^{-\varepsilon})$. Nevertheless, it is important to remark that finding $\Pi_{x}^f$ for a given function $f$ is a challenging task in most of the cases.

The following theorem recalls sufficient conditions to ensure that the continuity function is at least continuous.

\begin{theorem}[cf. \cite{CaLi1}]
\label{thm:continuity}
Consider $f:M_1\rightarrow M_2$ a non-constant, continuous function. Choose $x\in M_1$ and suppose that for any $r>0$ the closure of $B_{M_1}(x,r)$ is a compact set in $M_1$. Under these conditions, $\Pi^f_x$ is a continuous function.
\end{theorem}

As far as this work is concerned, sometimes it is convenient to understand $\Pi^f_x(\varepsilon)$ as a two parameter application given by $\Pi_f(x,\varepsilon)$ instead of a one parameter map. With this in mind, we propose an improvement of the results obtained so far.

\begin{definition}
Let $f:M_1\rightarrow M_2$ be a function such that $E_{f}(x) \not = \emptyset$ for any $x\in M_1$. Define
\[
    E_{f}(M_1)=\{(x,\varepsilon):x\in M_1\textrm{ and }\varepsilon\in E_{f}(x)\}
\]
and consider the two parameter continuity function $\Pi_{f}:  E_{f}(M_1) \rightarrow  (0,\infty)$ given by
\[
    \Pi_{f}(x,\varepsilon)=\Pi^{f}_x(\varepsilon).
\]
\end{definition}

The objective now is to prove that the aforementioned function is a continuous function in both variables. To this end, we first prove two fundamental results.

\begin{lemma}
\label{lem:aux}
Consider $\varepsilon>0$, a function $f:M_1\rightarrow M_2$ and a point $x\in M_1$. If $\delta\in\Delta_{f,x}(\varepsilon)$ and $z\in M_1$ is such that $d_1(x,z)<\delta$, then
\[
    \delta-d_1(x,z)\in\Delta_{f,z}(\varepsilon+d_2(f(x),f(z))).
\]
\end{lemma}

\begin{proof} Observe that if $y\in M_1$ and $d_1(y,z)<\delta-d_1(x,z)$, then we deduce by the Triangle Inequality that
\[
    d_1(x,y)<\delta.
\]

Thus, since $\delta\in\Delta_{f,x}(\varepsilon)$, it holds that $d_2(f(x),f(y))<\varepsilon$. Therewith
\[
    d_2(f(y),f(z))\leq d_2(f(y),f(x))+d_2(f(x),f(z))<\varepsilon+d_2(f(x),f(z)),
\]
what guarantees that $\delta-d_1(x,z)\in\Delta_{f,z}(\varepsilon+d_2(f(x),f(z))).$
\qquad \end{proof}

A consequence from the previous lemma can be stated as follows.

\begin{corollary}
\label{cor:aux2}
Let $\varepsilon>0$, $f:M_1\rightarrow M_2$, and assume that $x,z\in M_1$ satisfies $d_2(f(x),f(z))<\varepsilon$. If $\delta\in\Delta_{f,x}(\varepsilon-d_2(f(x),f(z)))$ and $d_1(x,z)<\delta$, then
\[
    \delta-d_1(x,z)\in\Delta_{f,z}(\varepsilon).
\]
\end{corollary}

We now address an important lemma that is used to prove the main theorem of this section.

\begin{lemma}
\label{lem:seq}
Assume that $f:M_1\rightarrow M_2$ is a continuous function. Let $\{x_n\}_{n=1}^\infty\subset M_1$ be a sequence converging to $x \in M_1$ and $\{\varepsilon_n\}_{n=1}^\infty$ a sequence of real numbers converging to $\varepsilon\in E_f(x)$ such that $\varepsilon_n\in E_f(x_n)$. For each $n\in\mathbb{N}$, define $\delta_n=\max\Delta_{f,x_n}(\varepsilon_n)$. Under these conditions, there exists $M>0$ such that $\delta_n\geq M$ for any $n\in\mathbb{N}$.
\end{lemma}

\begin{proof}
This result is proved by contradiction. Assume that there is a subsequence $\{\delta_{n_k}\}_{k=1}^\infty$ of $\{\delta_{n}\}_{n=1}^\infty$ such that $\delta_{n_k}\leq 1/k$ for each $k\in\mathbb{N}$. Thus, by the definition of $\delta_{n_k}$ and for each $k\in\mathbb{N}$, there exists $y_k\in M_1$ such that
\begin{equation}
\label{eq:seq01}
    \delta_{n_k}\leq d_1(x_{n_k},y_k)<\delta_{n_k} +(1/k)\textrm{ and } d_2(f(x_{n_k}),f(y_k))\geq\varepsilon_{n_k},
\end{equation}
since otherwise $\delta_{n_k}$ would not be the maximum of $\Delta_{f,x_{n_k}}(\varepsilon_{n_k})$. Now observe that by making $k\rightarrow\infty$ we can obtain from \eqref{eq:seq01} that
\begin{equation*}
    d_2(f(x),f(x))\geq\varepsilon>0,
\end{equation*}
which is a contradiction. This completes the proof of this result.
\qquad \end{proof}

We end this section by proving that the two parameter continuity function is a continuous mapping.

\begin{theorem}
Let $f:M_1\rightarrow M_2$ be a continuous function. If the one parameter continuity function is continuous, then the two parameter continuity function $\Pi_{f}:  E_{f}(M_1) \rightarrow  (0,\infty)$ is continuous.
\end{theorem}

\begin{proof}
Consider a point $(x,\varepsilon)\in E_{f}(M_1)$ and a sequence $\{(x_n,\varepsilon_n)\}_{n=1}^\infty$ in $E_{f}(M_1)$ which converges to $(x,\varepsilon)$. That is, $x_n\rightarrow x$ in $M_1$ and $\varepsilon_n\rightarrow \varepsilon$ in $\mathbb{R}$, when $n\rightarrow\infty$.

Set $\delta_n=\max{\Delta_{f,x_n}(\varepsilon_n)}=\Pi_{x_n}^f(\varepsilon_n)$. By Lemma \ref{lem:seq}, there exists $N_0\in\mathbb{N}$ such that $d_1(x_n,x)<\delta_n$ for any $n\geq N_0$. Since $\delta_n\in\Delta_{f,x_n}(\varepsilon_n)$, Lemma \ref{lem:aux} ensures that
\begin{equation}
\label{eq:above}
    \delta_n-d_1(x_n,x)\leq \Pi_x^f(\varepsilon_n+d_2(f(x_n),f(x))),
\end{equation}
for any $n\geq N_0$. By applying upper limit in both sides of inequality \eqref{eq:above} and using the theorem hypotheses about the one parameter continuity function, we obtain
\begin{equation}
\label{eq:main01}
    \limsup_{n\rightarrow\infty}{\Pi_{x_n}^f(\varepsilon_n)}\leq \Pi_x^f(\varepsilon).
\end{equation}

On the other side, by the continuity of $f$, there also exists $N_1\in\mathbb{N}$ such that
\[
    d_2(f(x),f(x_n))<\varepsilon_n
\] for any $n\geq N_1$. Define the positive real value
\[
    \bar{\delta}_n=\max{\Delta_{f,x}(\varepsilon_n-d_2(f(x),f(x_n)))}.
\]

By Lemma \ref{lem:seq} we choose $N_2\in\mathbb{N}$, if necessary, such that $d_1(x,x_n)<\bar{\delta}_n$ for any $n\geq N_2$. Therefore, Corollary \ref{cor:aux2} ensures that
\[
    \bar{\delta}_n-d_1(x,x_n)\leq \Pi_{x_n}^f(\varepsilon_n)
\]
for any $n\geq N_2$. Now applying the lower limit in both sides of the previous inequality, we obtain
\begin{equation}
\label{eq:main02}
    \Pi_x^f(\varepsilon)\leq \liminf_{n\rightarrow\infty}{\Pi_{x_n}^f(\varepsilon_n)}.
\end{equation}

It is now clear, by \eqref{eq:main01} and \eqref{eq:main02}, that  $\lim_{n\rightarrow\infty}{\Pi_{x_n}^f(\varepsilon_n)}=\Pi_x^f(\varepsilon)$, which implies that the two parameter continuity function is continuous.
\qquad \end{proof}

Observe that the previous theorem states a remarkable phenomenon about the continuity function. It is a well known fact that continuity in each variable is not enough for the global function to be continuous itself. However, this last theorem showcases that the continuity function does not suffer from this pathology.

\begin{theorem}
\label{thm:conclusao}
Consider $f:M_1\rightarrow M_2$ a non-constant, continuous function. Suppose that for any $r>0$ and $x \in M_1$ the closure of $B_{M_1}(x,r)$ is a compact set in $M_1$. Under these conditions, $\Pi_f$ is a continuous application.
\end{theorem}

\begin{proof}
By making use of Theorem \ref{thm:continuity}, we can finally state that the two parameters continuity function $\Pi_f$ is a continuous function itself, whenever $f$ is a continuous, non constant function and the closure of  $B_{M_1}(x,r)$ is a compact set in $M_1$.
\qquad \end{proof}

\begin{corollary}
If $f:M_1 \subset \mathbb{R} \rightarrow M_2 \subset \mathbb{R}$ is a non-constant, continuous function, then the two variable continuity function $\Pi_f$ is a continuous application. Besides that, the set
\[
    \{(x, \varepsilon, \delta) \in \mathbb{R}^3 : (x,\varepsilon) \in E_f \text{ and } \Pi_f(x,\varepsilon) = \delta\}
\] is a topological manifold homeomorphic to $E_f$.
\end{corollary}

We end this section with the formal definition of the surface associated to the two parameter continuity function.

\begin{definition}
The manifold defined by the aforementioned corollary is called the  \emph{$\varepsilon$--$\delta$ manifold} for $f$.
\end{definition}

\section{The Continuity Function as a Computable Diffeomorphism}
\label{sec:algorithm}

At this point it is known that there is a continuous function $\Pi_f$ which provides the maximal $\delta$ for the $\varepsilon$--$\delta$ relation for any given $f$ under the conditions of Theorem \ref{thm:conclusao}. We are now qualified to investigate in which circumstances the two parameter continuity function is a computable diffeomorphism. Let us begin with a definition.

\begin{definition}
\label{def:lagrange}
Given a function $f: M_1 \rightarrow M_2$ and a point $x \in M_1$, we are going to say that $f$ satisfies the \emph{Lagrange Propriety} at $x$ if the following two conditions hold:

\begin{itemize}
    \item[(i)] There exists a $C^1$ function $\Gamma: (-\zeta,\zeta) \subset \mathbb{R} \rightarrow \mathbb{R}$ such that \begin{equation*}\Gamma(r) = \sup_{y \in B_{M_1}[x,r]\setminus\{x\}} \frac{d_2(f(x), f(y))}{d_1(x,y)}\end{equation*} for all $r \in (0,\zeta)$;
    \item[(ii)] For all $r \in  (0,\zeta)$, there is an element $y_r \in  B_{M_1}[x,r]\setminus\{x\}$ such that \begin{equation*}\sup_{y \in B_{M_1}[x,r]\setminus\{x\}} \frac{d_2(f(x), f(y))}{d_1(x,y)}=\frac{d_2(f(x), f(y_r))}{d_1(x,y_r)}.\end{equation*}
\end{itemize}
\end{definition}

The result that follows presents a simple criterion for checking if a given function satisfies the Lagrange Propriety at a certain point.

\begin{lemma}
Let $U \subset \mathbb{R}$ be an open set. If $f: U \rightarrow \mathbb{R}$ is a $C^2$ function, then $f$ satisfies the Lagrange Propriety at $x \in U$ provided that $f'(x)f''(x) \not =0$.
\end{lemma}

\begin{proof}
Consider the auxiliary function given by
\[
    g(y) =
    \begin{cases}
        \displaystyle \frac{f(y) - f(x)}{y-x}, & \text{ if } y \not = x
    \\[10pt]
        f'(x), &\text { if } y = x
    \end{cases}
\]

Note that $g$ is a $C^1$ application such that $g'(x) = f''(x)/2 \not = 0$. Since $g'$ is continuous, then, by its signal conservation, $g$ is a strictly monotonic function in a certain neighborhood of $x$. Because $f'(x) \not = 0$,  then $g(x) \not =0$. Hence $g$ is a strictly monotonic function that does not change sign in a possibly smaller neighborhood.

Without loss of generality, assume that $g$ is positive and increasing. It is easy to see, under this assumptions, that
\[
    \sup_{y \in B_{M_1}[x,r]\setminus\{x\}} \frac{d_2(f(x), f(y))}{d_1(x,y)} =\sup_{y \in B_{M_1}[x,r]\setminus\{x\}} |g(y)| =g(x+r),
\]
for all sufficiently small $r$. Letting $\Gamma(r) = g(x+r)$, we conclude the first demand of the previous definition. The second one is obtained easily by noting that
\[
    g(x+r) = \frac{f(x+r) - f(x)}{(x+r) - x}.
\]
The other configurations for $g$ follows analogously.
\qquad \end{proof}

Let us emphasize that the hypotheses of last lemma does not settle a necessary condition for the Lagrange Propriety to be satisfied. For instance, affine functions are within the postulates prescribe by Definition \ref{def:lagrange}, although its second derivative is zero everywhere.

As proved in the literature, if $f: M_1 \rightarrow M_2$ satisfies the Lagrange Propriety, then the one parameter $\Pi_x^f$ gains a local boost in regularity provided that $\Gamma$ itself is regular. Besides that, the same work also proves that $\Pi_x^f$ must also satisfies an equation in terms of $\Gamma$. Let us recall this result here.

\begin{theorem}[cf. \cite{CaLi1}]
\label{thm:main_fact}
Suppose that $f: M_1 \rightarrow M_2$ is any given function that satisfies the Lagrange Propriety at $x\in M_1$. If $\Gamma(0) \not = 0$ and $\Gamma'(0) > 0$, then the continuity function of $f$ is a $C^k$ diffeomorphism in $(0,\varepsilon_0)$, provided that $\Gamma$ is a $C^k$ application. In this case, if $\varepsilon \in (0,\varepsilon_0)$ and if $\delta$ is such that
\begin{equation}
\label{eq:cruz}
    \varepsilon = \delta \Gamma(\delta) =: \Delta (\delta),
\end{equation}
then $\delta = \Pi_x^f(\varepsilon)$.
\end{theorem}

Last theorem ensures the following important result.

\begin{corollary}
Let $U \subset \mathbb{R}$ be an open set and $x\in U$. If $f: U \rightarrow \mathbb{R}$ is a $C^k$ function, $k\geq 2$, and $f'(x)f''(x) \not =0$, then the one parameter $\Pi_x^f$ is a $C^{k-1}$ diffeomorphism in a neighborhood of $x$. As before, if $\varepsilon \in (0,\varepsilon_0)$ and if $\delta$ is such that
\begin{equation*}
    \varepsilon = \delta \Gamma(\delta) =: \Delta (\delta),
\end{equation*}
then $\delta = \Pi_x^f(\varepsilon)$.
\end{corollary}

Observe that last corollary cannot be applied for the two parameter continuity function $\Pi_f$, since there is no hope for $\Pi_f$ to be a diffeomorphism due the distinct topological dimensions of the domain and the image. It is also important to remark that the regularity gain is local. We know that $\Pi_f$ is globally continuous, but Theorem \ref{thm:main_fact} is of local nature --- that is, we obtain that $\Pi_x^f$ is a $C^k$ diffeomorphism in a open set which may be properly contained in its domain.

Let us now focus in Equation \eqref{eq:cruz}, which is the crux of the matter. To calculate $\delta=\Pi_f(x,\varepsilon)$ for a certain function $f$, we have to solve the equation $\varepsilon=\delta \Gamma(\delta)$. If $f$ is a real function, acting on one real variable, we have that \eqref{eq:cruz} reduces itself to
\begin{equation}
\label{eq:cruz_matter}
    \varepsilon =  \delta \max_{0 < | x -y| <\delta} \frac{|f(x)-f(y)|}{|x-y|}.
\end{equation}

Recall that for \eqref{eq:cruz_matter} to be valid, we need $f: M_1 \subset \mathbb{R} \rightarrow M_2 \subset \mathbb{R}$ to satisfy the Lagrange Propriety at $x$ in such a way that $\Gamma(0) \not = 0$ and $\Gamma'(0) > 0$. This set of hypotheses composes our basic \emph{theoretical assumptions} on $f$ from now on. It is important to keep in mind that these premises are satisfied for all $C^2$ functions defined on open subsets such that $f'(x)f''(x) \not =0$.

To proof that \eqref{eq:cruz_matter} is computable, there are two tasks to be done. The first one is to find out $\Delta$, while the second one is to solve $\varepsilon  = \Delta(\delta)$ for $\delta$. However, to show that \eqref{eq:cruz_matter} is computable, the basic theoretical assumptions are not enough. We need extra \emph{Turing assumptions} on $f$ for correct implementation, and those requirements will be described on time, as we make use of them.

There are numerous combinations of algorithms that are able to handle the tasks proposed, each one with its own hypotheses, convergence speed and error control techniques. Depending of the particular properties of the given function, one may be more suitable than the other.

We now propose an easy to implement, divide--and--conquer algorithm that computes \eqref{eq:cruz_matter} for a significantly large class of functions. To begin with, let $x$ and $\varepsilon$ be the parameters in which we are interested in. Before this discussion, let us fix some definitions and notations that are constantly used from this point on.

\begin{definition}
\label{def:leibniz}
Let $f$ be a real function defined on some subset of the real line.  Define the \emph{Leibniz ratio} of $f$ around $x$ as
\[
    \mathcal{L}_xf(y) =  \frac{|f(x)-f(y)|}{|x-y|}
\] for all $x$ and $y$ where the previous formula is well defined. When there is no possibility of confusion, we drop the subscript that indicates the point $x$.
\end{definition}

\begin{definition}
A real function $f$ of one real variable is called \emph{unimodal} if there is one single local maximum value for $f$.
\end{definition}

In the light of the Definition \ref{def:leibniz}, \eqref{eq:cruz_matter} is now reduced to
\[
    \varepsilon =  \delta \max_{0 < | x -y| <\delta} \mathcal{L}_x f(y) = \Delta(\delta).
\]

Our first objective now is to prove that $\Delta$ is computable. For this, we will assume that $\mathcal{L}_x$ is an unimodal and Lipschitz continuous function. This new set of hypotheses will make our \emph{Turing assumptions} on f. As before, let us further investigate sufficient requirements on $f$ for $\mathcal{L}_x$ to fit the aforementioned hypothesis.

\begin{lemma}
\label{lem:lipschitz}
Let $U \subset \mathbb{R}$ be an open set and $K \subset U$ be a compact interval. If $f: U \rightarrow \mathbb{R}$ is a $C^2$ function and $x \in K$, then there is a constant $M= M(f,K)$ such that
\[
    |\mathcal{L}_x f(y) - \mathcal{L}_x f(z)| \leq M  |y-z|,
\]
for all $y,z \in K\setminus \{ x\}$.
\end{lemma}

\begin{proof}
Consider the auxiliary function given by
\[
    g(y) =
    \begin{cases}
        \displaystyle \frac{f(y) - f(x)}{y-x}, & \text{ if } y \not = x
    \\[10pt]
        f'(x), &\text { if } y = x
    \end{cases}
\]

It is easy to see that $g \in C^1$ and that $\mathcal{L}_x f(y) = | g(y) |$. Since $K$ is a compact set, let $M$ be the maximum of $|g'|$ over it. Using the Reverse Triangular Inequality and the Mean Value Theorem, we obtain
\[
    |\mathcal{L}_x f(y) - \mathcal{L}_x f(z)| = | | g(y)| - |g(z) | | \leq | g(y) - g(z) | \leq  M  |y-z|.
\]

The last inequality completes the proof of this lemma.
\qquad \end{proof}

\begin{definition}
Let $f: U \rightarrow \mathbb{R}$ be a differentiable function defined on some open subset of the real line and fix some $x \in U$.  The function $f$ will be called of \emph{transversal type} at $x$ if the equation
\[
    f'(y) = \frac{f(y) - f(x)}{y - x}
\]
has only a finite number of solutions.
\end{definition}

The previous definition states a geometrical imposition. Shortly, it says that the secant passing through $(x,f(x))$ and $(y,f(y))$ can match the tangent at $(y,f(y))$ for at most a finite number of $y$'s. For instance, an affine function is not of transversal type at any point. The importance of such requirement is elucidated by the next result.

\begin{lemma}
Let $f: U \rightarrow \mathbb{R}$ be a $C^2$ function of transversal type at $x$. Under these conditions, $\mathcal{L}_x f$ is unimodal at a certain closed non-degenerated neighborhood of $x$.
\end{lemma}

\begin{proof}
Let $g$ be as in the proof of Lemma \ref{lem:lipschitz}. If we show that $g$ has a finite number of extrema points, we can isolate them in neighborhoods and complete the proof. Since $g$ is a $C^1$ mapping, it is easy to see that $g'$ has a finite number of roots, once $f$ is of transversal type at $x$. Hence $g$ has at most a finite number of critical points.
\qquad \end{proof}

Putting it all together, we achieve the following result.

\begin{theorem}
Let $f: U \rightarrow \mathbb{R}$ be a $C^2$ function of transversal type at $x$. If $f'(x) f'' (x) \not = 0$, then $f$ satisfies both the theoretical and Turing assumptions in a neighborhood of $x$.
\end{theorem}

From our previous discussion, assume that $\mathcal{L}f$ is an unimodal Lipschitz function with Lipschitz constant $M>0$. Let us begin with a standard ternary search to find the maximum value of $\mathcal{L}f$ and consequently prove that $\Delta$ is computable. Assume that  $\omega_\Delta$ is the desired precision. We finish the proof if there is an algorithm, which output will be designed by $\tilde{\Delta}$, such that
\[
    |\tilde{\Delta}(\delta) - \Delta(\delta)| < \omega_\Delta,
\] for all meaningful $\delta$.

Here follows an outline for $\tilde{\Delta}$. Recall that unimodality ensures the existence and uniqueness of a single maximum value in its domain. Since $\mathcal{L}f$ is an unimodal function by hypotheses, we know that the only maximum of the function lies in a certain $[a,b]$. Let $y^*$ be a point in $[a,b]$ such that $\mathcal{L}f(y^*)$ reaches its unique maximum value.

Assume that $\omega_{\operatorname{sup}}$ denotes a constant that handle the ternary search precision. The value of this constant will be properly chosen in function of $\omega_\Delta$.

At this point, consider successively smaller refinements of the interval $[a,b]$, which will be called $[a_j, b_j]$, such that $y^* \in [a_j, b_j] \subset [a_{j-1}, b_{j-1}] \subset [a,b]$ and $|b_j - a_j| \leq \frac{2}{3} |b_{j-1} - a_{j-1}|.$  Once $[a_j, b_j]$ is sufficiently small, we may approximate $y^*$ by one of the interval extreme points. Consider a refinement such that $|b_j - a_j| < \omega_{\operatorname{sup}}/ M$. Assume that we are choosing $\xi \in [a_j, b_j]$ as an approximation for the maximum point. Since $\xi$ and $y^*$ lies in the same set $[a_j, b_j]$, then
\[
    |\mathcal{L}f(y^*) - \mathcal{L}f(\xi)| \leq  M |y^* - \xi| \leq M |b_j - a_j | \leq \omega_{\operatorname{sup}}.
\]

\begin{algorithm}[ht]
\label{alg:ternary}
\KwData{$\mathcal{L}f$, $\omega_{\operatorname{sup}}$, $M$, $a$, $b$}
\KwResult{$\sup_{[a,b]} \mathcal{L}f$, with precision $\omega_{\operatorname{sup}}$}
initialization point\;
\eIf{$|a-b|<\omega_{\operatorname{sup}}/ M$}
{
    \tcc{there is no need for extra interval refinement}
    return $\max\{\mathcal{L}f(a), \mathcal{L}f(b)\}$\;
}
{
    \tcc{pick 2 equidistant points marking 1/3 of the interval}
    set $p = a + (b - a) / 3$\;
    set $q = b - (b - a) / 3$\;
    set $\gamma_p= \mathcal{L}f(p)$\;
    set $\gamma_q = \mathcal{L}f(q)$\;
    \tcc{since $\mathcal{L}f$ is unimodal, the following applies}
    \uIf{$\gamma_p < \gamma_q$}
    {
        \tcc{the maximum point must be in $[q, b]$}
        restart with $a= q$ unless exceeded the number of iterations\;
    }
    {
    \uElseIf{$\gamma_p > \gamma_q$}
    {
        \tcc{the maximum point must be in $[a, p]$}
        restart with $b= p$ unless exceeded the number of iterations\;
    }
    \Else{
        \tcc{the maximum point must be in $[p, q]$}
         restart with $a= p$ and $b=q$ unless exceeded the number of iterations\;
    }
    }
}
\caption{Ternary search for $\Gamma$ computation}
\end{algorithm}

It is easy to see that Algorithm \ref{alg:ternary} (cf. page \pageref{alg:ternary}) enable us to calculate $\tilde{\Delta}(\delta)$ by invoking it at the interval $[x - \delta, x + \delta]$ and multiplying its output by $\delta$. Also note that this procedure has runtime order $\Theta(\log n)$. Because of the final output multiplication, we get that
\[
    |\tilde{\Delta}(\delta) - \Delta(\delta)| < \delta \omega_{\operatorname{sup}}.
\]

Since we need $\delta \omega_{\operatorname{sup}} \leq \omega_\Delta$ to obtain $|\tilde{\Delta}(\delta) - \Delta(\delta)| < \omega_\Delta$, we must have some control over $\delta$ so we can properly choose $\omega_{\operatorname{sup}}$. To begin this discussion we state the following corollary from Theorem \ref{thm:main_fact}.

\begin{corollary}
\label{cor:interval}
Suppose that $\mathcal{B}_1$ and $\mathcal{B}_2$ are Banach spaces. In addition to Theorem \ref{thm:main_fact} hypothesis, suppose that $M_1$ is an open set of $\mathcal{B}_1$ and also assume that $M_2 \subset \mathcal{B}_2$. If $f$ is differentiable and $L$ is the maximum value of function $t\mapsto\|f'(t)\|_{\mathcal{L}(B_1,B_2)}$, then \
\begin{equation*}
    \frac{\varepsilon}{L} \leq \Pi_x^f(\varepsilon) \leq \frac{\varepsilon}{\Gamma(0)}
\end{equation*}  for all sufficiently small $\varepsilon$. If $t\mapsto \|f'(t)\|_{\mathcal{L}(B_1,B_2)}$ does not reach a maximum value, then the first inequality is reduced to $0 \leq \Pi_x^f(\varepsilon)$.
\end{corollary}

\begin{proof}
The first inequality is a trivial consequence of the Mean Value Theorem together with the fact that $\Pi_x^f$ provides the maximum suitable number for the $(x,\varepsilon)$ parameters. To proof the second one, note that $\Gamma$ is a monotone application for all sufficiently small $\delta$. Hence $\Gamma(0) \leq \Gamma(\delta)$ and
\[
\Pi_x^f(\varepsilon) = \frac{\varepsilon}{\Gamma(\delta)} \leq \frac{\varepsilon}{\Gamma(0)},
\]
where $\delta = \Pi_x^f(\varepsilon)$.
\qquad \end{proof}

Recalling our main discussion, this last result ensures that if we choose
\[
    \omega_{\operatorname{sup}} < \omega_\Delta \frac{\Gamma(0)}{\varepsilon},
\]
then we have that $|\tilde{\Delta}(\delta) - \Delta(\delta)| < \omega_\Delta$, which completes the computability proof of $\Delta$.

In the particular case of real functions of one real variable, it is easy to calculate $\Gamma(0)$. For instance, assume that $f$ is differentiable. Then for all $s>0$, by the Mean Value Theorem, there is a  $\theta_s$ such that $| \theta_s - x| < s$ and $\mathcal{L}f(s) = |f'(\theta_s)|$. Making $s \rightarrow 0$ we get that \[\mathcal{L}f(0) = |f'(x)|.\] The same reasoning applies for $\Gamma(0)$.

Since $\Delta$ is now computable, we are in the conditions to solve the equation $\varepsilon = \Delta(\delta)$ for $\delta$. By employing Bolzano's Theorem, we use a binary search to look for the solution inside the interval $[a,b]$ settled by Corollary \ref{cor:interval}.

To begin with, recall that $\Delta = \Pi_x^{-1}$ is a homeomorphism. This allow us to denote by $\delta^*$ the unique solution of $\varepsilon = \Delta(\delta)$ that lies inside $[a,b]$. As done before, we proceed recursively, creating successively smaller refinements of the interval $[a,b]$, which are called $[a_j, b_j]$, such that $\delta^* \in [a_j, b_j] \subset [a_{j-1}, b_{j-1}] \subset [a,b]$ and $|b_j - a_j| \leq \frac{1}{2} |b_{j-1} - a_{j-1}|$.

Again, once $[a_j, b_j]$ is sufficiently small, we may approximate $\delta^*$ by one of the interval extreme points. Let $\omega_{\operatorname{sol}}$ be the given precision for the solution finding algorithm. Consider a refinement such that $|b_j - a_j| < \omega_{\operatorname{sol}}$. Assume that $\varsigma \in [a_j, b_j]$ is an approximation for the solution. Since $\varsigma$ and $\delta^*$ lies in the same set $[a_j, b_j]$, then $|\varsigma - \delta^*| < |b_j - a_j| < \omega_{\operatorname{sol}}.$

Now consider the pseudocode written bellow.

\begin{algorithm}[H]
\label{alg:binary}
\KwData{$\Delta$, $\varepsilon$, $\omega_{\operatorname{sol}}$, $a$, $b$}
\KwResult{$\delta$ such that $\Delta(\delta) = \varepsilon$, with precision $\omega_{\operatorname{sol}}$}
initialization point\;
\eIf{$|a-b|<\omega_{\operatorname{sol}}$}
{
    \tcc{there is no need for extra interval refinement}
    \eIf{$|\Delta(a) - \varepsilon| < |\Delta(b) - \varepsilon|$}
    {
        \tcc{$a$ is a better approximation than $b$}
        return $a$\;
    }
    {
        return $b$\;
    }
}
{
    \tcc{Since $\Delta$ is continuous, we use Bolzano's Theorem}
    set $m = (a + b)/2$\;
    set $\gamma_a = \Delta(a)$\;
    set $\gamma_b = \Delta(b)$\;
    set $\gamma_m = \Delta(m)$\;
    \eIf{$(\gamma_a - \varepsilon) (\gamma_m - \varepsilon)\leq0$}
    {
        \tcc{the solution must be in $[a,m]$}
         restart with $b= m$ unless exceeded the number of iterations\;
    }
    {
        \tcc{the solution must be in $[m,b]$}
        restart with $a= m$ unless exceeded the number of iterations\;
    }
}
\caption{Binary search for $\Delta(\delta)=\varepsilon$ solution }
\end{algorithm}

It is easy to see that Algorithm \ref{alg:binary} has runtime order $\Theta(\log n)$. Together with the first step, we managed to build an algorithm, with polylogarithm runtime order, for solving the $\varepsilon$--$\delta$ determination problem.

Note that the error is controlled by two independent parameters $\omega_\Delta$ and $\omega_{\operatorname{sol}}$. It is important to mention that for the fully precision control of this method, we need to calculate $M$ and $L$ for each function we are analyzing.

It is also important to note that $\omega_\Delta$ must be significantly smaller than $\omega_{\operatorname{sol}}$ for the composed algorithm to work properly. Putting it all together, we managed to prove the following couple results.

\begin{theorem}[Local Smoothness and Computability]
\label{thm:computability}
Let $f: U  \subset \mathbb{R} \rightarrow \mathbb{R}$ be a $C^k$ function, $k \geq 2$, of transversal type at $x$. If $f'(x) f'' (x) \not = 0$, then there is a function $\Pi_x^f: (0, \varepsilon_0) \rightarrow \mathbb{R}^+$ such that
\begin{enumerate}
    \item[(i)] $\Pi_x^f$ is a $C^{k-1}$ diffeomorphism over its image;
    \item[(ii)] $\Pi_x^f$ is a  monotonically increasing function;
    \item[(iii)] if $\delta = \Pi_x^f(\varepsilon)$ and $|x-y| < \delta$, then $|f(x) - f(y)| < \varepsilon$;
    \item[(iv)] $\Pi_x^f$ provides the maximum $\delta$ for which (iii) is valid;
    \item[(v)] $\Pi_x^f$ is computable, provided that $f$ and $x$ are computable.
\end{enumerate}
\end{theorem}

\begin{theorem}[Non-local Continuity] Let $f: U \subset \mathbb{R} \rightarrow \mathbb{R}$ be a non-constant, continuous function. Then there is an $\Omega \subset U \times \mathbb{R}^+$ and function $\Pi_f: \Omega \rightarrow \mathbb{R}^+$ such that
\begin{enumerate}
    \item[(i)] $\Pi_f$ is a continuous application;
    \item[(ii)] if $\delta = \Pi_f(x,\varepsilon)$ and $|x-y| < \delta$, then $|f(x) - f(y)| < \varepsilon$;
    \item[(iii)] $\Pi_f$ provides the maximum $\delta$ for which (ii) is valid;
    \item[(iv)] $\Omega$ is homeomorphic to the $\varepsilon$--$\delta$ manifold of $f$.
\end{enumerate}
\end{theorem}

\section{Final Discussions and Examples}
\label{sec:conclusions}

Since we have already presented the fundamental theory and the main algorithm that are the very soul of this work, we are now under conditions to address some examples in this final section. Nevertheless, it is imperative to stress that the cases  portrayed here makes up only a small portion of the functions for which these techniques applies.

Initially, we present functions from distinct classes of mappings that were selected from the crowd because they feature explicit $\Pi_x^f$ formulas, making it possible to validate the output of the algorithm. After that, we present a final example where the continuity function is not explicitly known, which drives us to check Theorem \ref{thm:computability} hypotheses. 

At this point, a general remark is in order. In the proof of the computability of the continuity function, we manage to prove that $\Pi_x^f$ is computable for sufficiently small $\varepsilon$. However, when the examples portrayed are  compared to their $\Pi_x^f$ formula counterpart, it showcases the accuracy and effectiveness of the algorithm even for larger values of $\varepsilon$.

\begin{example}
\label{ex:exponential}
Exponential functions.
\end{example}

Here we furnish a particular case of an exponential function. Let $\mathbb{R}$ be considered with the standard Euclidean metric and assume that $f:\mathbb{R}\rightarrow \mathbb{R}$ is given by
\[
     f(y)=1-e^{-y}.
\]

Let us initially concentrate our efforts to discuss the continuity function associated to $f$ in a theoretical point of view. By the definitions introduced in Section \ref{sec:theory}, it is not difficult to verify that $\Delta_{f,x}(\varepsilon)=(0,{x}+\ln{(\varepsilon+e^{-{x}})}]$ for any $x\in\mathbb{R}$. Following the procedure to deduce the continuity function, we observe that $E_f(x)=(0,\infty)$ and therefore the continuity function $\Pi_x^f:(0,\infty)\rightarrow\mathbb{R}$ is given by the formula
\begin{equation}
    \label{eq:pif}
    \Pi^f_{x}(\varepsilon)={x}+\ln{(\varepsilon+e^{-{x}})}.
\end{equation}

It is important to observe that changes in the value of $x$, however small, induces an entirely different continuity function. Following the ideas addressed in Section \ref{sec:theory}, it is also not difficult to notice that $E_f(\mathbb{R})=\mathbb{R}\times(0,\infty)$ and that the two parameter continuity function $\Pi_f:\mathbb{R}\times(0,\infty)\rightarrow\mathbb{R}^+$ is given by
\[
    \Pi_f(x,\varepsilon)=x+\ln{(\varepsilon+e^{-x})}.
\]

By implementing the algorithm discussed in Section \ref{sec:algorithm}, we can numerically compute the two parameter continuity function and graph the $\varepsilon$--$\delta$ manifold of $f$.

\begin{example}
Rational functions.
\label{ex:rational}
\end{example}

Consider $M=\mathbb{R}\setminus\{30\}$ with the induced Euclidean metric and $\mathbb{R}$ itself with the canonical metric. Let $f:M\rightarrow \mathbb{R}$ be given by
\[
    f(y)=\dfrac{1}{y-30}.
\]

Notice that this map has a completely distinguished behavior when compared with the function discussed in Example \ref{ex:exponential}. Nevertheless, observe that for each fixed $x\in\mathbb{R}$ it holds that $E_f(x)=(0,\infty)$. Calculating the one parameter continuity function $\Pi_{x}^f:(0,\infty)\rightarrow \mathbb{R}^+$, we obtain that
\begin{equation}
\label{eq:pirationalleq}
    \Pi^f_{x}(\varepsilon)= \frac{\varepsilon(x-30)^2}{1-\varepsilon(x-30)}
\end{equation}
for $x < 30$. On the other hand, if $x>30$, then
\begin{equation}
\label{eq:pirationalgeq}
    \Pi^f_{x}(\varepsilon) = \frac{\varepsilon(x -30)^2}{1+\epsilon(x-30)}.
\end{equation}

Two interesting phenomena needs to be clarified. First, we stress that the singular behavior of this continuity function is expected, since the original mapping contains a first order pole itself. Second, it is also noteworthy that the variable change $x + x'= 60 $ allow us to obtain the right hand side of \eqref{eq:pirationalgeq} from \eqref{eq:pirationalleq}. This is a reflex of the shifted parity of $f$, i.e.
\[
    f(30 - y) + f(30 + y) = 0.
\]

Like before, following the ideas addressed in Section \ref{sec:theory}, we obtain that $E_f(\mathbb{R})=\mathbb{R}\times(0,\infty)$ and that $\Pi_f:\mathbb{R}\times(0,\infty)\rightarrow\mathbb{R}^+$ is given by the right hand side of \eqref{eq:pirationalleq} and \eqref{eq:pirationalgeq}. Precisely,
\[
    \Pi_f(x,\varepsilon)=
    \begin{cases}
        \displaystyle \frac{\varepsilon(x-30)^2}{1-\varepsilon(x-30)}, & \text{ if } x < 30
    \\[10pt]
        \displaystyle \frac{\varepsilon(x -30)^2}{1+\epsilon(x-30)}, & \text{ if } x > 30.
    \end{cases}
\]

\begin{example}
Affine functions.
\label{ex:affine}
\end{example}

As outlined in Section \ref{sec:algorithm}, the conditions in which we stated our theorems are only sufficient ones. Therefore, as can be verified in this case, non-constant affine functions are the simplest examples in which the functions are outside the specified conditions and the proposed algorithm still forges a correct answer. To analyze a concrete case, let $\mathbb{R}$ be considered with the standard Euclidean metric and  $f:\mathbb{R}\rightarrow \mathbb{R}$ be given by
\[
    f(y)=2y+1.
\]

Observe that for each fixed $x\in\mathbb{R}$ it holds that $E_f(x)=(0,\infty)$, what ensures that the one parameter continuity function $\Pi_{x}^f:(0,\infty)\rightarrow \mathbb{R}^+$ is given by
\[
    \Pi^f_{x}(\varepsilon)=\varepsilon/2.
\]

Note that the maximum $\delta$ for this function is uniformly determined with respect to $x$. This is a remarkable fact that is not shared --- in general --- by any other uniformly continuous functions. In other words, this example states that there is a theoretical gap between uniform continuity and an uniform maximal $\delta$ for each $\varepsilon$.

It is easy to see that $E_f(\mathbb{R})=\mathbb{R}\times(0,\infty)$ and that the two parameter continuity function $\Pi_f:\mathbb{R}\times(0,\infty)\rightarrow\mathbb{R}^+$ is given by
\[
    \Pi_f(x,\varepsilon)=\varepsilon/2.
\]

To conclude this paper, we exhibit a situation that is slightly different in nature when compared to previous cases: $\Pi_f$ is unknown due to the lack of tools for solving \eqref{eq:cruz_matter}.

\begin{example}
\label{ex:quadratic}
A function for which $\Pi_x^f$ is not explicitly known.
\end{example}

Consider $\mathbb{R}$ with the canonical Euclidian distance and let $f: \mathbb{R} \rightarrow \mathbb{R}$ be given by
\[
    f(y) = y^2 + 11y.
\]

This is the simplest polynomial for which $\Pi_x^f$ is yet unknown. Because we do not have an explicit continuity function in hands, we must carefully check our algorithm hypotheses to numerically find it. Initially note that if $f'(x) f''(x) = 0$, then $x = -11/2$. Hence, if we restrict  $x$ to $[-5,5]$, then $f$ satisfies the theoretical assumptions of our work. Now note that
\begin{equation}
\label{eq:trans}
f'(y) = \frac{f(y) - f(x)}{y - x} \iff 0= \frac{y^2 - 2xy + x^2}{y-x}.
\end{equation}

Therefore $f$ is of transversal type at any $x$, since for each fixed $x$ the only solution for \eqref{eq:trans} belongs to $\{y \in \mathbb{R}: y=x\}$, which is a finite set. Thus $f$ satisfies the Turing assumptions of our theory. Consequently, we may freely apply the algorithm proposed in Section \ref{sec:algorithm} to find the continuity function for $f$ at any $x \in [-5,5]$. 

A final remark is in order. In Examples \ref{ex:exponential}, \ref{ex:rational} and \ref{ex:affine} the continuity function $\Pi_x^f$ was explicitly known, and solving \eqref{eq:cruz_matter} for each case was not a hard endeavor. However, Example \ref{ex:quadratic} showed us that this is not necessarily a standard fact. As we could see, there weretrans many issues to find the continuity function formula related to $f(y)=y^2 + 11y$. The conjecture is that this difficulty is intimately connected with the absence of a bijective property.

\section*{Acknowledgements}
On the occasion of the preparation of this manuscript, the second author would like to thank the Federal University of Santa Catarina (UFSC) for the hospitality and support during a short term visit in Florian\'{o}polis. The second author has been partially supported by CAPES (process PNPD 2770/2011).

\end{document}